\documentclass[11pt]{article}
\usepackage[margin=1.25in]{geometry}
\usepackage{amsmath,amssymb,amsthm}
\usepackage{tikz}
\usetikzlibrary{arrows.meta,decorations.markings,calc,angles,quotes}
\usepackage{subcaption}

\newtheorem{theorem}{Theorem}
\newtheorem{proposition}{Proposition}
\theoremstyle{remark}
\newtheorem*{remark}{Remark}

\DeclareMathOperator{\Area}{Area}
\newcommand{\dang}{\measuredangle}

\title{Signed Quantities and Quadrilaterals:\\ Geometry without Cases}
\author{George-Rubin Thomas\\[2pt]
\small Founder, MathCamp, MathPath, Epsilon Camp, and ProofZoom\\
\small \texttt{thomas@proofzoom.org}}
\date{}

\begin{document}
\maketitle

\begin{abstract}
Ordinary angles and unsigned areas discard orientation, and much of the
case analysis in plane geometry is the price of that loss. Signed
quantities (directed angles, oriented areas and directed distances) keep
the signs and so let one argument serve for every configuration. We
illustrate this with results that hold uniformly for convex, concave and
crossed quadrilaterals: the diagonal area formula, an identity expressing
one diagonal in terms of the sides and the other diagonal, a test for a
pair of parallel sides, a theorem on equidiagonal quadrilaterals (such a
quadrilateral has a pair of equal opposite sides if and only if it has a
pair of parallel opposite sides), and a single ratio of directed distances
that decides whether the vertices lie on an ellipse, only on hyperbolas, or
on a circle.
\end{abstract}

\section{Introduction}

Is the area of a quadrilateral half the product of its diagonals times the
sine of the angle between them? For a convex quadrilateral it is. Harries
\cite{Harries} showed that the same expression $\tfrac12 ef\sin\theta$ also
gives the area of a concave quadrilateral, by a separate computation, and
that for a crossed quadrilateral a third computation turns it into the
\emph{difference} of the areas of the two triangular lobes. Arguments of this
kind, split three ways into convex, concave and crossed cases, are common in
the metric geometry of quadrilaterals. The split rarely reflects the
geometry itself. It comes from the signs lost when we pass to ordinary
angles, unsigned areas and unsigned distances, and it disappears when those
signs are restored.

The idea is old. Johnson~\cite{Johnson} introduced directed angles between
lines in 1917 precisely so that theorems such as Simson's could be stated
and proved once for every figure, instead of for whichever figure happened
to be drawn. Here we apply the same idea to metric results about
quadrilaterals, using three signed quantities, which Section~2 introduces:
directed angles, oriented areas and directed distances. Each result below
holds, with one proof, for convex, concave and crossed quadrilaterals alike.

Section~3 returns to the area formula: with oriented area it takes one
line, and it shows at once why the crossed case gives a difference.
Section~4 expresses one diagonal in terms of the four sides, the other
diagonal and the oriented areas of the two triangles on that diagonal
(Theorem~\ref{thm:diag}). The proof is the classical law-of-cosines argument
for a convex quadrilateral, written with directed angles; with no other
change it becomes valid for every configuration. Applying the identity to
both diagonals gives a simple test for a pair of parallel sides
(Proposition~\ref{prop:trap}) and a relation valid for every quadrilateral
(Proposition~\ref{prop:both}). Section~5 specializes this relation to
quadrilaterals with equal diagonals and proves that such a quadrilateral has
a pair of equal opposite sides if and only if it has a pair of parallel
opposite sides (Theorem~\ref{thm:equi}); for a convex quadrilateral this is
the familiar fact that a trapezoid with equal diagonals is isosceles.
Section~6 turns to the conics through the four vertices. A single ratio $k$
of directed distances decides whether they lie on an ellipse ($k>0$) or
only on hyperbolas ($k<0$), and $k=1$ exactly when they lie on a circle
(Theorem~\ref{thm:conic}).

The tools are elementary: coordinates, determinants, the law of cosines and
the power of a point. We hope that the paper will be useful to teachers and
their students as an illustration of how much a well-chosen sign can save.

\section{Directed angles, oriented areas and directed distances}

\subsection{Directed angles}

An ordinary angle $\angle BAC$ records the separation of two rays but not
their order: $\angle BAC=\angle CAB$. Johnson's directed angle between two
lines $l$, $m$ is the counterclockwise rotation carrying $l$ onto $m$. It is
defined modulo $\pi$, because a half-turn carries a line onto itself. This is
the natural tool for incidence questions such as concyclicity, but it cannot
supply the sines and cosines that metric identities require, since
$\sin(\theta+\pi)=-\sin\theta$.

We therefore use the corresponding angle between vectors. For nonzero
vectors $u$, $v$, let
\[
\dang(u,v)\in\mathbb{R}/2\pi\mathbb{Z}
\]
be the counterclockwise rotation from the direction of $u$ to the direction
of $v$. Reducing it modulo $\pi$ gives Johnson's angle between the
underlying lines. Its sine and cosine are well defined, and
\[
u\cdot v=|u|\,|v|\cos\dang(u,v),\qquad
\det(u,v)=|u|\,|v|\sin\dang(u,v),
\]
where $\det(u,v)=u_1v_2-u_2v_1$. Directed angles obey
\[
\dang(v,u)=-\dang(u,v),\qquad \dang(u,w)=\dang(u,v)+\dang(v,w)
\pmod{2\pi},
\]
with no exceptions for the relative positions of $u$, $v$, $w$. For points
we write $\dang BAC:=\dang(B-A,\,C-A)$. Because cosine is even, the law of
cosines holds verbatim with directed angles:
$BC^2=AB^2+AC^2-2\,AB\cdot AC\cos\dang BAC$.

\subsection{Oriented area}

For a polygon with vertices $A_i=(x_i,y_i)$ listed in order, with
$A_{n+1}=A_1$, the oriented area is given by the shoelace formula
\[
[A_1\cdots A_n]=\frac12\sum_{i=1}^n \det(A_i,A_{i+1})
=\frac12\sum_{i=1}^n (x_iy_{i+1}-x_{i+1}y_i).
\]
Reversing the order of the vertices reverses the sign. More precisely, the
polygon divides the plane into regions, and $[A_1\cdots A_n]$ is the sum of
their areas, each weighted by the number of times the polygon winds
counterclockwise around it. For a simple polygon this is $\pm$ the ordinary
area, with $+$ for counterclockwise traversal. For a crossed quadrilateral
the two lobes are wound in opposite senses and enter with opposite signs.

For a triangle the definition reduces to
\[
[ABC]=\tfrac12\det(B-A,\,C-A)=\tfrac12\,AB\cdot AC\,\sin\dang BAC .
\]
Thus $|[ABC]|$ is the ordinary area, and the sign of $[ABC]$ records the
orientation of the ordered triple $A,B,C$.

\subsection{Directed distances}

Two kinds of directed distance appear in Section~6. On a line with a chosen
direction, the directed length $\overline{XY}$ is $\pm XY$, positive when
$Y-X$ points along the chosen direction. For points $O,X,Y$ on one line the
product
\[
\overline{OX}\cdot\overline{OY}=(X-O)\cdot(Y-O)
\]
does not depend on the chosen direction, and it is negative exactly when $O$
lies between $X$ and $Y$. This is the setting of the power of a point: if
two lines through $O$ meet a circle at $X,Y$ and at $Z,W$, then
$\overline{OX}\cdot\overline{OY}=\overline{OZ}\cdot\overline{OW}$, whether
$O$ is inside or outside the circle; conversely, if $X,Y$ and $Z,W$ lie on
two different lines through $O$, all distinct from $O$, and this equation
holds, then $X,Y,Z,W$ are concyclic.

The directed distance from a point $X$ to the line $PQ$ is
\[
h_X(PQ)=\frac{\det(Q-P,\,X-P)}{PQ}=\frac{2[PQX]}{PQ},
\]
so that an oriented area is half a base times a directed height:
$[PQX]=\tfrac12\,PQ\cdot h_X(PQ)$. Reversing the direction of $PQ$ changes
the sign of $h_X(PQ)$, so a product $h_X(PQ)\,h_Y(PQ)$ does not depend on it;
the product is positive when $X$ and $Y$ lie on the same side of the line and
negative when they lie on opposite sides.

\section{Area from the diagonals}

As recalled in the Introduction, Harries~\cite{Harries} obtains the formula
$\tfrac12 ef\sin\theta$ for convex, concave and crossed quadrilaterals by
three separate computations. Oriented area explains why the crossed answer
is a difference, and it gives all three cases in a single line.

Let $A,B,C,D$ be listed in their order of traversal. Expanding,
\begin{align*}
\det(C-A,\,D-B)&=\det(C,D)-\det(C,B)-\det(A,D)+\det(A,B)\\
&=\det(A,B)+\det(B,C)+\det(C,D)+\det(D,A),
\end{align*}
so the shoelace formula gives
\[
[ABCD]=\tfrac12\det(C-A,\,D-B).
\]
Let $e=AC$, $f=BD$ and $\theta=\dang(C-A,\,D-B)$, the directed angle from
the vector $C-A$ to the vector $D-B$. Then
\[
[ABCD]=\tfrac12\,ef\sin\theta .
\]
In particular, the sign of $\sin\theta$ is the orientation of the
traversal $ABCD$.

For a convex or concave quadrilateral, $\Area(ABCD)=|[ABCD]|=\tfrac12
ef\,|\sin\theta|$. For a crossed quadrilateral the two lobes are traversed
in opposite senses, so $|[ABCD]|$ is the difference of their areas, exactly
as Harries found. Their total area is the sum, which the diagonal formula
does not give.

\section{A diagonal identity without cases}

\subsection{The two triangles on a diagonal}

Let $ABCD$ be a quadrilateral with $A\ne C$, and put
\[
\Delta_1^e=[ABC],\qquad \Delta_2^e=[ADC].
\]
Both ordered triangles traverse their common side from $C$ to $A$. Hence
they have opposite orientations if $B$ and $D$ lie on opposite sides of the
line $AC$, and the same orientation if $B$ and $D$ lie on the same side:
\[
\Delta_1^e\Delta_2^e<0\ \text{for opposite-side gluing},\qquad
\Delta_1^e\Delta_2^e>0\ \text{for same-side gluing}.
\]
In a convex quadrilateral both diagonals give opposite-side gluing; in a
crossed quadrilateral both give same-side gluing. In a concave
quadrilateral the gluing is opposite-side for the diagonal through the
reflex vertex and same-side for the other diagonal (Figure~\ref{fig:glue}).

\begin{figure}[htbp]
\centering
\tikzset{
  pt/.style={circle,fill,inner sep=1.1pt},
  mid/.style={postaction={decorate},decoration={markings,
    mark=at position 0.55 with {\arrow{Stealth[length=5pt]}}}},
  diag/.style={dashed,mid},
}
\begin{subfigure}{0.31\textwidth}\centering
\begin{tikzpicture}[scale=1.05]
  \coordinate (A) at (0,0); \coordinate (B) at (1.3,0.75);
  \coordinate (C) at (0,2.1); \coordinate (D) at (-1.35,0.8);
  \draw[mid] (A)--(B); \draw[mid] (B)--(C);
  \draw[mid] (A)--(D); \draw[mid] (D)--(C);
  \draw[diag] (C)--(A);
  \draw[-{Stealth[length=3.5pt]}] ($(A)!0.42!(B)$) arc[start angle=30,end angle=90,radius=0.6];
  \node at (0.28,0.75) {\scriptsize$\alpha$};
  \draw[-{Stealth[length=3.5pt]}] ($(A)!0.38!(D)$) arc[start angle=149,end angle=90,radius=0.6];
  \node at (-0.28,0.75) {\scriptsize$\delta$};
  \draw[-{Stealth[length=3.5pt]}] ($(C)!0.32!(B)$) arc[start angle=-46,end angle=-90,radius=0.6];
  \node at (0.24,1.35) {\scriptsize$\beta$};
  \draw[-{Stealth[length=3.5pt]}] ($(C)!0.32!(D)$) arc[start angle=-135,end angle=-90,radius=0.6];
  \node at (-0.24,1.35) {\scriptsize$\gamma$};
  \foreach \P/\pos in {A/below,B/right,C/above,D/left} \node[pt,label=\pos:$\P$] at (\P) {};
\end{tikzpicture}
\caption{Convex}
\end{subfigure}\hfill
\begin{subfigure}{0.31\textwidth}\centering
\begin{tikzpicture}[scale=1.05]
  \coordinate (A) at (0,0.7); \coordinate (B) at (1.35,-0.2);
  \coordinate (C) at (0,2.4); \coordinate (D) at (-1.35,-0.2);
  \draw[mid] (A)--(B); \draw[mid] (B)--(C);
  \draw[mid] (A)--(D); \draw[mid] (D)--(C);
  \draw[diag] (C)--(A);
  \foreach \P/\pos in {A/below,B/right,C/above,D/left} \node[pt,label=\pos:$\P$] at (\P) {};
\end{tikzpicture}
\caption{Concave}
\end{subfigure}\hfill
\begin{subfigure}{0.31\textwidth}\centering
\begin{tikzpicture}[scale=1.05]
  \coordinate (A) at (0,0); \coordinate (B) at (1.4,1.6);
  \coordinate (C) at (0,2.2); \coordinate (D) at (1.4,0.35);
  \draw[mid] (A)--(B); \draw[mid] (B)--(C);
  \draw[mid] (A)--(D); \draw[mid] (D)--(C);
  \draw[diag] (C)--(A);
  \foreach \P/\pos in {A/left,B/right,C/left,D/right} \node[pt,label=\pos:$\P$] at (\P) {};
\end{tikzpicture}
\caption{Crossed}
\end{subfigure}
\caption{The ordered triangles $ABC$ and $ADC$ glued along the diagonal
$AC$. Arrows show the traversals $A\to B\to C\to A$ and $A\to D\to C\to A$.
In (a) and (b), $B$ and $D$ lie on opposite sides of $AC$ and the triangles
have opposite orientations; in (c) they lie on the same side and the
triangles have the same orientation. The arcs in (a) show the directed
angles used in the proof of Theorem~\ref{thm:diag}.}
\label{fig:glue}
\end{figure}
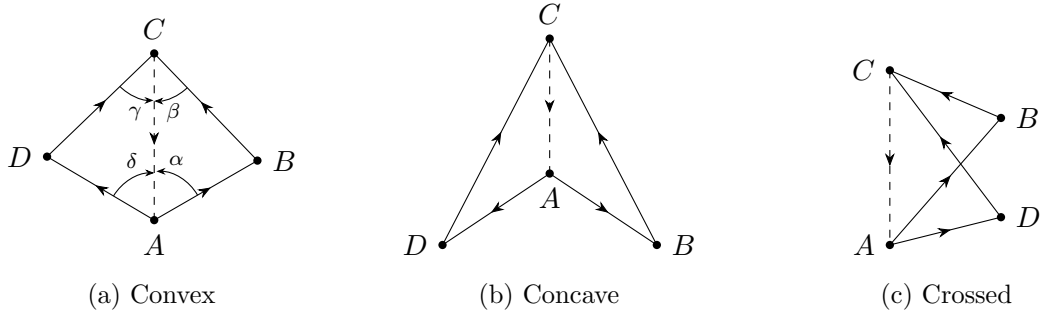

\subsection{The identity}

\begin{theorem}[Quadrilateral diagonal identity]\label{thm:diag}
Let $ABCD$ be a quadrilateral with sides $AB=a$, $BC=b$, $CD=c$, $DA=d$ and
diagonals $AC=e$, $BD=f$, where $A\ne C$. Then
\begin{equation}\label{eq:diag}
2f^2=a^2+b^2+c^2+d^2-e^2+\frac{(a^2-b^2)(c^2-d^2)}{e^2}
-\frac{16\,\Delta_1^e\Delta_2^e}{e^2}.
\end{equation}
\end{theorem}

\begin{proof}
Suppose first that $B$ and $D$ are distinct from $A$ and $C$, so that
$a,b,c,d>0$. Measure every angle from the diagonal: put
\[
\alpha=\dang BAC,\quad \delta=\dang DAC,\quad
\beta=\dang BCA,\quad \gamma=\dang DCA .
\]
By the addition rule, $\dang BAD=\dang BAC+\dang CAD=\alpha-\delta$, and
likewise $\dang BCD=\beta-\gamma$. The law of cosines in triangles $ABD$ and
$CBD$ gives
\begin{equation}\label{eq:cos}
f^2=a^2+d^2-2ad\cos(\alpha-\delta),\qquad
f^2=b^2+c^2-2bc\cos(\beta-\gamma).
\end{equation}
In the triangles $ABC$ and $ADC$ on the diagonal,
\begin{align*}
\cos\alpha&=\frac{a^2+e^2-b^2}{2ae}, & \cos\delta&=\frac{d^2+e^2-c^2}{2de},\\
\cos\beta&=\frac{b^2+e^2-a^2}{2be}, & \cos\gamma&=\frac{c^2+e^2-d^2}{2ce},
\end{align*}
while the sines are given by the oriented areas:
\[
\Delta_1^e=[ABC]=\tfrac12 ae\sin\alpha,\qquad
\Delta_2^e=[ADC]=\tfrac12 de\sin\delta,
\]
and, reading the same triangles from $C$,
$\Delta_1^e=[CAB]=-\tfrac12 be\sin\beta$ and
$\Delta_2^e=[CAD]=-\tfrac12 ce\sin\gamma$. Hence
\[
\sin\alpha\sin\delta=\frac{4\Delta_1^e\Delta_2^e}{ade^2},\qquad
\sin\beta\sin\gamma=\frac{4\Delta_1^e\Delta_2^e}{bce^2}.
\]
Expanding the cosines of differences,
\begin{align*}
2ad\cos(\alpha-\delta)&=\frac{(a^2+e^2-b^2)(d^2+e^2-c^2)}{2e^2}
+\frac{8\Delta_1^e\Delta_2^e}{e^2},\\
2bc\cos(\beta-\gamma)&=\frac{(b^2+e^2-a^2)(c^2+e^2-d^2)}{2e^2}
+\frac{8\Delta_1^e\Delta_2^e}{e^2}.
\end{align*}
Writing $P=a^2-b^2$ and $Q=d^2-c^2$, the two products are $(e^2+P)(e^2+Q)$
and $(e^2-P)(e^2-Q)$, whose sum is $2e^4+2PQ$. Adding the two equations
\eqref{eq:cos} therefore gives
\[
2f^2=a^2+b^2+c^2+d^2-e^2-\frac{(a^2-b^2)(d^2-c^2)}{e^2}
-\frac{16\Delta_1^e\Delta_2^e}{e^2},
\]
which is \eqref{eq:diag}.

Finally, after multiplication by $e^2$ both sides of \eqref{eq:diag} are
polynomials in the coordinates of $A,B,C,D$. They agree whenever $B$ and $D$
avoid $A$ and $C$, hence, by continuity, for all positions of $B$ and $D$.
\end{proof}

\begin{remark}
For a convex quadrilateral this is the classical argument, with ordinary
angles: $\angle BAD=\angle BAC+\angle CAD$, every sine is positive, and
$\sin\angle BAC\,\sin\angle CAD$ is a positive multiple of the product of
the two triangle areas. For a crossed quadrilateral the classical argument
must instead write $\angle BAD=|\angle BAC-\angle CAD|$, and the sign of
the area term changes; the concave case needs further attention, since the
two diagonals behave differently. With directed angles ($\alpha$ and
$\delta$), the relation $\dang BAD=\alpha-\delta$ holds in every
configuration (and likewise $\dang BCD=\beta-\gamma$), and the sign of
$\Delta_1^e\Delta_2^e$ records the configuration automatically. The proof
needs no other change.
\end{remark}

\subsection{The radical form and its two roots}

The usual formula for the area of a triangle in terms of its sides,
$16\Delta^2=4a^2e^2-(a^2+e^2-b^2)^2$, turns \eqref{eq:diag} into
\begin{multline}\label{eq:rad}
2f^2=a^2+b^2+c^2+d^2-e^2+\frac{(a^2-b^2)(c^2-d^2)}{e^2}\\
\pm\frac{1}{e^2}\sqrt{\bigl[4a^2e^2-(a^2+e^2-b^2)^2\bigr]
\bigl[4d^2e^2-(d^2+e^2-c^2)^2\bigr]},
\end{multline}
with the plus sign for opposite-side gluing and the minus sign for
same-side gluing. For convex quadrilaterals, where the gluing is always
opposite-side, the formula with the plus sign was obtained by Sedrakyan and
Mozayeni~\cite{SM}; \eqref{eq:diag} extends it to concave and crossed
quadrilaterals. The radical form displays the dependence on the five lengths
$a,b,c,d,e$; the oriented-area form \eqref{eq:diag} carries the sign.

The two signs have a simple meaning. Reflecting $D$ in the line $AC$ to a
point $D'$ leaves $a,b,c,d,e$ unchanged but reverses the sign of
$\Delta_2^e$. So the two values (of $f^2$) in \eqref{eq:rad} are $BD^2$ and
$BD'^2$: the (remaining) five lengths fix the two triangles on $AC$ but not
the side of $AC$ on which they are glued. Equivalently, the six distances
among four coplanar points satisfy a single polynomial relation, the
vanishing of their Cayley--Menger determinant~\cite{Blumenthal}. As an equation in $f^2$
that relation is quadratic, and its two roots are the two values in
\eqref{eq:rad}. Identity~\eqref{eq:diag} says which root is the actual
diagonal: the sign of $\Delta_1^e\Delta_2^e$ selects it.

\begin{remark}
Write $L=(a^2+b^2+c^2+d^2-e^2)\,e^2+(a^2-b^2)(c^2-d^2)$, so that
\eqref{eq:diag} reads $2e^2f^2-L=-16\,\Delta_1^e\Delta_2^e$. Expanding the
Cayley--Menger determinant of $A,B,C,D$,
\[
\mathrm{CM}=\det\begin{pmatrix}
0&1&1&1&1\\
1&0&a^2&e^2&d^2\\
1&a^2&0&b^2&f^2\\
1&e^2&b^2&0&c^2\\
1&d^2&f^2&c^2&0
\end{pmatrix},
\]
one finds
\[
\mathrm{CM}=-\frac{1}{2e^2}\Bigl[(2e^2f^2-L)^2-256\,(\Delta_1^e)^2(\Delta_2^e)^2\Bigr].
\]
The squares $(\Delta_1^e)^2$ and $(\Delta_2^e)^2$ depend only on lengths,
so the planar relation $\mathrm{CM}=0$ determines $2e^2f^2-L$ only up to
sign. Identity \eqref{eq:diag} is its signed square root: orientation
supplies the sign, much as the shoelace formula supplies the sign that
Heron's formula loses.

Conversely, the factorization yields a second proof of \eqref{eq:diag}. Put
$G=2e^2f^2-L+16\Delta_1^e\Delta_2^e$ and $H=2e^2f^2-L-16\Delta_1^e\Delta_2^e$.
Both are polynomials in the coordinates of $A,B,C,D$, and
$\mathrm{CM}=-GH/(2e^2)$ vanishes for every plane configuration, so $GH$ is
the zero polynomial and one of $G$, $H$ vanishes identically. For the unit
square $e^2=f^2=2$, $L=4$ and $\Delta_1^e\Delta_2^e=-\tfrac14$, so $G=0$ and
$H=8$; hence $G\equiv0$, which is \eqref{eq:diag}. This proof is short once
the determinant has been expanded, but the sign comes from an algebraic
argument and a single example; in the proof of Theorem~\ref{thm:diag} it
comes from the geometry.

The picture becomes concrete if the quadrilateral is folded along $AC$, so
that the triangles $ABC$ and $ADC$ meet at a dihedral angle $\psi$. Then
\eqref{eq:diag} still holds with $\Delta_1^e\Delta_2^e$ replaced by
$|\Delta_1^e|\,|\Delta_2^e|\cos\psi$, and $\mathrm{CM}=288V^2$, where $V$
is the volume of the tetrahedron $ABCD$. As $\psi$ runs from $0$ to
$\pi$, $f^2$ moves between the two values in \eqref{eq:rad}. The flat
positions $\psi=0$ and $\psi=\pi$ are exactly those where the volume
vanishes, and there $\cos\psi=\pm1$ is the sign of
$\Delta_1^e\Delta_2^e$.
\end{remark}

\subsection{Both diagonals together}

Apply \eqref{eq:diag} to the diagonal $BD$ as well, by the cyclic relabelling
$A\to B\to C\to D\to A$. The sides become $b,c,d,a$, the diagonals exchange
roles, and the two triangles on $BD$ are
\[
\Delta_1^f=[BCD],\qquad \Delta_2^f=[BAD].
\]
Multiply the two instances of \eqref{eq:diag} by $e^2$ and $f^2$
respectively and subtract. Since
$(a^2-b^2)(c^2-d^2)-(b^2-c^2)(d^2-a^2)=(a^2-c^2)(b^2-d^2)$, we obtain, for
every quadrilateral with $A\ne C$ and $B\ne D$,
\begin{equation}\label{eq:gen}
(a^2-c^2)(b^2-d^2)-16\bigl(\Delta_1^e\Delta_2^e-\Delta_1^f\Delta_2^f\bigr)
=(e^2-f^2)\bigl(e^2+f^2-a^2-b^2-c^2-d^2\bigr).
\end{equation}

Both sides of \eqref{eq:gen} have a direct geometric meaning. For the area
term on the left, write
\[
p=[ABC],\quad q=[ACD],\quad r=[BCD],\quad s=[BDA],
\]
so that $\Delta_1^e=p$, $\Delta_2^e=-q$, $\Delta_1^f=r$ and
$\Delta_2^f=-s$. Each diagonal cuts the polygon $ABCD$ into two triangles,
so $p+q=r+s=[ABCD]$; call this common value $K$. Substituting $q=K-p$ and
$s=K-r$,
\[
\Delta_1^e\Delta_2^e-\Delta_1^f\Delta_2^f=rs-pq=(r-p)(K-r-p)=-(r-p)(r-q).
\]
Each factor is a single determinant:
\begin{align*}
r-p&=[BCD]-[BCA]=\tfrac12\det(C-B,\,D-A),\\
r-q&=[CDB]-[CDA]=\tfrac12\det(D-C,\,B-A).
\end{align*}
Hence
\begin{equation}\label{eq:area}
16\bigl(\Delta_1^e\Delta_2^e-\Delta_1^f\Delta_2^f\bigr)
=-4\det(C-B,\,D-A)\,\det(D-C,\,B-A).
\end{equation}
The first determinant vanishes exactly when $BC\parallel AD$, and the second
exactly when $CD\parallel AB$. This gives a case-free criterion for a pair of
parallel sides.

\begin{proposition}\label{prop:trap}
For every quadrilateral $ABCD$ with $A\ne C$ and $B\ne D$,
$\Delta_1^e\Delta_2^e=\Delta_1^f\Delta_2^f$ if and only if $ABCD$ has a pair
of parallel opposite sides.
\end{proposition}

Since the sums $p+q$ and $r+s$ are always equal, the equal products mean
that $\{p,q\}=\{r,s\}$ as unordered pairs: either $p=r$ and $q=s$, which
holds exactly when $BC\parallel AD$, or $p=s$ and $q=r$, which holds exactly
when $AB\parallel CD$.

For example, in a parallelogram each diagonal cuts the area $K$ into two
halves, so both products equal $-K^2/4$, although the diagonals are equal
only for a rectangle. The products detect parallel sides, not equal
diagonals.

For the right side, let $m$ be the distance between the midpoints of the
diagonals. Then $2m=|(A+C)-(B+D)|$, and expanding $|A+C-B-D|^2$ gives
Euler's quadrilateral relation $a^2+b^2+c^2+d^2=e^2+f^2+4m^2$. Combining
this with \eqref{eq:gen} and \eqref{eq:area} gives the following.

\begin{proposition}\label{prop:both}
For every quadrilateral $ABCD$ with $A\ne C$ and $B\ne D$,
\begin{equation}\label{eq:full}
(a^2-c^2)(b^2-d^2)+4\det(C-B,\,D-A)\,\det(D-C,\,B-A)=-4m^2(e^2-f^2),
\end{equation}
where $m$ is the distance between the midpoints of the diagonals.
\end{proposition}

\section{Equidiagonal quadrilaterals}

Setting $e=f$ in Proposition~\ref{prop:both}, the right side of
\eqref{eq:full} vanishes, and we obtain the following.

\begin{theorem}\label{thm:equi}
Let $ABCD$ be a quadrilateral, convex, concave or crossed, with equal
diagonals $AC=BD$. Then
\[
(a^2-c^2)(b^2-d^2)=-4\det(C-B,\,D-A)\,\det(D-C,\,B-A).
\]
Consequently, $ABCD$ has a pair of equal opposite sides ($a=c$ or $b=d$) if
and only if it has a pair of parallel opposite sides ($AB\parallel CD$ or
$BC\parallel AD$).
\end{theorem}

Here ``parallel'' includes the degenerate case in which the two sides lie on
one line. For a convex quadrilateral the theorem contains the familiar fact
that a trapezoid with equal diagonals is isosceles, together with its
converse; the proof above makes no use of convexity.

\section{Conics through the four vertices}

\subsection{Background}

Every triangle has a circumcircle, and any five points lie on a conic. Four
points with no three collinear lie on infinitely many conics, so the
natural question is which kinds of conic occur. It is well known that the
four points lie on an ellipse exactly when they are in convex position; Li,
Wang, Shen and Dai~\cite{Li} give a proof of existence by parallel
projection. The ellipses through the vertices of a convex quadrilateral
have a long history. Steiner characterised the most nearly circular one
(see D\"orrie~\cite[Problem~54]{Dorrie}); Horwitz~\cite{Horwitz2010}
proved that there are unique ellipses of minimal area and of minimal
eccentricity through the four vertices, and his book~\cite{HorwitzBook}
collects these and related results; and Silvester~\cite{Silvester}
extended to convex quadrilaterals, in the \emph{Mathematical Gazette}, the classical
results on ellipses of extremal area associated with a triangle.

Our contribution is to show that the type of conic is decided by a single
number $k$, a ratio of products of directed distances, and that the same
number detects when the four points lie on a circle. We begin with an
example.

\subsection{An example}

Take the diagonals along the coordinate axes, meeting at the origin $O$:
\[
A=(-2,0),\quad C=(1,0),\quad B=(0,1),\quad D=(0,-3).
\]
We look for a conic through the four points with no $xy$ term,
\[
x^2+ky^2+px+qy+r=0 .
\]
On the $x$-axis it reduces to $x^2+px+r=0$, which must have the roots $-2$
and $1$; so $r=(-2)(1)=-2$ and $p=1$. On the $y$-axis it reduces to
$ky^2+qy+r=0$, which must have the roots $1$ and $-3$; so
$r=k\cdot(1)(-3)$, giving $k=\tfrac23$, and $q=\tfrac43$. Clearing
fractions, the conic is
\[
3x^2+2y^2+3x+4y-6=0,
\]
and one checks at once that all four points satisfy it. The coefficients
of $x^2$ and $y^2$ have the same sign, so after completing the squares,
$3(x+\tfrac12)^2+2(y+1)^2=\tfrac{35}{4}$, it is an ellipse.

Now move $B$ to $B'=(0,-1)$, inside the triangle $ACD$, so that $AB'CD$ is
concave. The roots on the $y$-axis become $-1$ and $-3$, so $r=-2=3k$ and
$k=-\tfrac23$. The conic is now
\[
3x^2-2y^2+3x-8y-6=0,\qquad\text{that is,}\qquad
2(y+2)^2-3(x+\tfrac12)^2=\tfrac54,
\]
a hyperbola (Figure~\ref{fig:conics}).

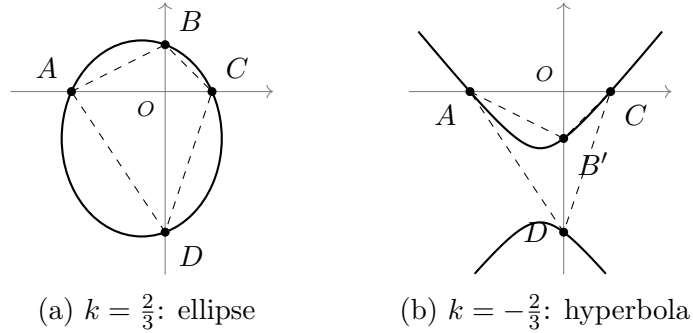
\begin{figure}[htbp]
\centering
\begin{tikzpicture}[scale=0.62]
  \draw[->,gray] (-3.3,0)--(2.3,0); \draw[->,gray] (0,-3.9)--(0,1.9);
  \draw[thick] (-0.5,-1) ellipse ({sqrt(35/12)} and {sqrt(35/8)});
  \draw[dashed] (-2,0)--(0,1)--(1,0)--(0,-3)--cycle;
  \foreach \P/\x/\y/\pos in {A/-2/0/above left,C/1/0/above right,B/0/1/above right,D/0/-3/below right}
    \node[circle,fill,inner sep=1.2pt,label=\pos:$\P$] at (\x,\y) {};
  \node[below left] at (0,0) {\scriptsize $O$};
  \node at (-0.4,-4.7) {(a) $k=\tfrac23$: ellipse};
\end{tikzpicture}\qquad\qquad
\begin{tikzpicture}[scale=0.62]
  \draw[->,gray] (-3.3,0)--(2.3,0); \draw[->,gray] (0,-3.9)--(0,1.9);
  \draw[thick,domain=-3.1:2.1,samples=60] plot (\x,{-2+sqrt((1.25+3*(\x+0.5)^2)/2)});
  \draw[thick,domain=-1.9:0.9,samples=60] plot (\x,{-2-sqrt((1.25+3*(\x+0.5)^2)/2)});
  \draw[dashed] (-2,0)--(0,-1)--(1,0)--(0,-3)--cycle;
  \foreach \P/\x/\y/\pos in {A/-2/0/below left,C/1/0/below right,B'/0/-1/below right,D/0/-3/left}
    \node[circle,fill,inner sep=1.2pt,label=\pos:$\P$] at (\x,\y) {};
  \node[above left] at (0,0) {\scriptsize $O$};
  \node at (-0.4,-4.7) {(b) $k=-\tfrac23$: hyperbola};
\end{tikzpicture}
\caption{The conic with no $xy$ term through the vertices of (a) a convex
and (b) a concave quadrilateral whose diagonals lie along the axes.}
\label{fig:conics}
\end{figure}

The example contains the whole idea. The constant $r$ is the product of the
two roots on each axis, so
\[
r=\overline{OA}\cdot\overline{OC}=k\cdot\overline{OB}\cdot\overline{OD},
\qquad\text{that is,}\qquad
k=\frac{\overline{OA}\cdot\overline{OC}}{\overline{OB}\cdot\overline{OD}} .
\]
By \S2.3, each product of directed lengths is negative exactly when $O$
lies between the two points. In the convex case $O$ lies inside both
diagonals, both products are negative, and $k>0$; in the concave case $O$
lies inside one diagonal and outside the other, and $k<0$. The sign of $k$
then decides whether $x^2$ and $y^2$ have coefficients of the same or
opposite sign.

\subsection{The ratio $k$ in general}

For a general quadrilateral the diagonals need not be perpendicular, and
they may even be parallel. We therefore define $k$ by directed distances,
which make sense in every case. For a quadrilateral $ABCD$ with no three
vertices collinear, put
\begin{equation}\label{eq:k}
k=\frac{h_A(BD)\,h_C(BD)}{h_B(AC)\,h_D(AC)}
=\frac{e^2\,\Delta_1^f\Delta_2^f}{f^2\,\Delta_1^e\Delta_2^e} .
\end{equation}
Thus $k$ compares, with signs, how far the vertices off each diagonal lie
from it. The second form follows from $[PQX]=\tfrac12\,PQ\cdot h_X(PQ)$
(\S2.3): each triangle on $AC$ has base $e$ and each triangle on $BD$ has
base $f$, so
\[
\Delta_1^e\Delta_2^e=\tfrac14e^2\,h_B(AC)\,h_D(AC),\qquad
\Delta_1^f\Delta_2^f=\tfrac14f^2\,h_A(BD)\,h_C(BD).
\]

When the diagonals meet at a point $O$, the example's formula returns. Let
$\varphi$ be the angle between the diagonals. A point $X$ on the line $AC$
has directed distance $\overline{OX}\,\sin\varphi$ from the line $BD$, up to
a sign that is the same for every such $X$; and similarly for points on
$BD$. In each product of two distances the sign and the factor
$\sin^2\varphi$ are common to numerator and denominator, and cancel:
\begin{equation}\label{eq:kchord}
k=\frac{\overline{OA}\cdot\overline{OC}\,\sin^2\varphi}
{\overline{OB}\cdot\overline{OD}\,\sin^2\varphi}
=\frac{\overline{OA}\cdot\overline{OC}}{\overline{OB}\cdot\overline{OD}} .
\end{equation}
So $k$ is a ratio of directed lengths along the diagonals, a ratio of
directed distances from the diagonals, and a ratio of oriented areas, all
at once.

\subsection{The theorem}

\begin{theorem}\label{thm:conic}
Let $ABCD$ be a quadrilateral with no three vertices collinear. Then
\begin{enumerate}
\item[(i)] $k>0$ if and only if $A,B,C,D$ lie on an ellipse; this is the
  case for convex and crossed quadrilaterals;
\item[(ii)] $k<0$ if and only if every conic through $A,B,C,D$, other than
  a pair of lines, is a hyperbola; this is the case for concave
  quadrilaterals;
\item[(iii)] $A,B,C,D$ lie on a circle if and only if $k=1$ when the
  diagonals $AC$ and $BD$ are not parallel, and if and only if $BC=DA$ when
  they are.
\end{enumerate}
\end{theorem}

\begin{proof}
\emph{The sign of $k$.} If the diagonals meet at $O$, \eqref{eq:kchord} and
\S2.3 show, as in the example, that $k>0$ when $O$ lies inside both
diagonals (convex) or outside both (crossed), and $k<0$ when it lies inside
one and outside the other (concave). If $AC\parallel BD$, both products of
distances in \eqref{eq:k} equal the square of the distance between the two
lines, so $k=1>0$. In this case $ABCD$ is necessarily crossed, since the
diagonals of a convex or concave quadrilateral lie on intersecting lines;
but its vertices lie two on each of two parallel lines, so taken in another
order they form a trapezoid, and the four points are in convex position. In oriented
areas, the sign of $k$ is the sign of
$\Delta_1^e\Delta_2^e\cdot\Delta_1^f\Delta_2^f$, positive exactly when the
two diagonals give the same type of gluing (\S4.1).

\emph{No ellipse for a concave quadrilateral.} Say $D$ lies inside the
triangle $ABC$. An ellipse bounds a convex region; if it passes through
$A$, $B$, $C$, that region contains the whole triangle, so $D$ lies strictly
inside the ellipse and not on it. The same holds for a parabola. So the only
conics through the four points, apart from pairs of lines, are hyperbolas.

\emph{An ellipse for a convex quadrilateral.} The four points of a crossed
quadrilateral are those of a convex one, taken in a different order, so we
may assume $ABCD$ convex; its diagonals meet at a point $O$ inside both.
If they are perpendicular, take them as axes and repeat the example: the
conic $x^2+ky^2+px+qy+r=0$ with $r=\overline{OA}\cdot\overline{OC}$ and
$k=(\overline{OA}\cdot\overline{OC})/(\overline{OB}\cdot\overline{OD})>0$
passes through the four points, and since $k>0$ it is an ellipse. If they
are not perpendicular, first apply a shear that fixes the line $AC$ and
turns the line $BD$ until it is perpendicular to $AC$. A shear maps lines
to lines and preserves the order of points on a line, so the image is again
a convex quadrilateral and has an ellipse through its vertices; and the
inverse shear maps that ellipse to an ellipse through $A,B,C,D$.

\emph{Circles.} When the diagonals meet at $O$, the power of the point $O$
and its converse (\S2.3) say that $A,B,C,D$ are concyclic exactly when
$\overline{OA}\cdot\overline{OC}=\overline{OB}\cdot\overline{OD}$, that is,
by \eqref{eq:kchord}, exactly when $k=1$. In the example this is visible
directly: when the diagonals are perpendicular and $k=1$, the conic
$x^2+y^2+px+qy+r=0$ is a circle. If $AC\parallel BD$, the four points form a
trapezoid with parallel sides $AC$ and $BD$ and diagonals $AD$ and $BC$; it
is cyclic exactly when it is isosceles, that is, by Section~5, exactly when
its diagonals are equal: $BC=DA$.
\end{proof}

Part (iii) is a criterion for concyclicity that, like the other results in
this paper, needs no distinction between convex, concave and crossed
quadrilaterals. When the diagonals are not parallel, $ABCD$ is cyclic
exactly when $e^2\,\Delta_1^f\Delta_2^f=f^2\,\Delta_1^e\Delta_2^e$, or,
equivalently, when the product of the directed distances from $A$ and $C$
to the line $BD$ equals the product of the directed distances from $B$ and
$D$ to the line $AC$. When $AC\parallel BD$ the ratio $k$ is always $1$ and
gives no information; that is exactly the trapezoid case, and it is
settled by the equidiagonal theorem of Section~5.

\subsection{The size of $k$}

Every positive value of $k$ occurs: if $O$ is the midpoint of both
diagonals, with $OA=OC=1$ and $OB=OD=t$, then $k=1/t^2$. Relabelling
$A\to B\to C\to D\to A$ exchanges the diagonals and replaces $k$ by $1/k$,
so the cyclic value $k=1$ sits in the middle. At the extremes, $k\to\infty$
when $B$ or $D$ approaches the line $AC$, that is, when the quadrilateral
flattens towards a triangle with a vertex on the line of a diagonal. In the
example, letting $B=(0,\varepsilon)$ with $\varepsilon\to0$ gives
$k=2/(3\varepsilon)\to\infty$: the ellipse stretches without bound along the
$x$-axis while its height stays bounded, and in the limit it becomes the
pair of parallel lines $y=0$ and $y=-3$, the line through the three
collinear points and the parallel line through the fourth. Symmetrically,
$k\to0$ when $A$ or $C$ approaches the line $BD$.

For readers who know conjugate diameters: the ellipse we have constructed
has no $xy$ term in coordinates along the diagonals, so the directions of
$AC$ and $BD$ are conjugate directions of it, and it is the only conic
through $A,B,C,D$ with this property. Completing the squares shows that $k$
is the square of the ratio of its semi-diameters parallel to $AC$ and to
$BD$; the cyclic case $k=1$ is the one in which these two conjugate
semi-diameters are equal.

\section{Conclusion}

In each example the signs that ordinary angles, unsigned areas and
unsigned distances discard are carried by signed quantities, and one
argument covers convex, concave and crossed quadrilaterals alike. Oriented area gives
the diagonal area formula in one line and explains Harries's crossed case.
Directed angles let the classical convex proof of the diagonal identity
stand unchanged for every configuration. Applied to both diagonals, that
identity yields a relation valid for every quadrilateral, and from it a
case-free theorem on equidiagonal quadrilaterals. Finally, directed
distances turn the same oriented areas into a single ratio $k$ that decides
which conics pass through the four vertices, and in particular gives a
case-free criterion for a quadrilateral to be cyclic.
The case-by-case
treatments remain useful when one wants to see each configuration on its
own, but they are no longer necessary.

\section*{Use of artificial intelligence}

The author used Claude (Anthropic), an AI assistant, to help review and
edit drafts, check calculations numerically, prepare the figures and
\LaTeX{} source, and search the literature. The author reviewed all of its
contributions and takes full responsibility for the content of the paper.


{\small
\begin{thebibliography}{99}
\bibitem{Harries} J. Harries, Note 86.50: Area of a quadrilateral,
\emph{Math. Gaz.} \textbf{86} (July 2002) pp.~310--311.
\bibitem{Johnson} R. A. Johnson, Directed angles in elementary geometry,
\emph{Amer. Math. Monthly} \textbf{24} (1917) pp.~101--105.
\bibitem{SM} H. Sedrakyan and A. Mozayeni, Formulas for diagonals of any
quadrilateral, \emph{Mathematical Reflections} \textbf{4} (2022).
\bibitem{Blumenthal} L. M. Blumenthal, \emph{Theory and applications of
distance geometry}, Clarendon Press, Oxford (1953).
\bibitem{Li} J. H. Li, Z. Q. Wang, Y. X. Shen and Z. Y. Dai, Does any
convex quadrilateral have circumscribed ellipses?, \emph{Open Math.}
\textbf{15} (2017) pp.~1463--1476.
\bibitem{Dorrie} H. D\"orrie, \emph{100 great problems of elementary
mathematics}, trans. D. Antin, Dover, New York (1965).
\bibitem{Horwitz2010} A. Horwitz, Ellipses of minimal area and of minimal
eccentricity circumscribed about a convex quadrilateral, \emph{Aust. J.
Math. Anal. Appl.} \textbf{7} (2010) Article~8.
\bibitem{HorwitzBook} A. Horwitz, \emph{Ellipses inscribed in, and
circumscribed about, quadrilaterals}, Chapman \& Hall/CRC (2024).
\bibitem{Silvester} J. R. Silvester, Extremal area ellipses of a convex
quadrilateral, \emph{Math. Gaz.} \textbf{101} (March 2017) pp.~11--26.
\end{thebibliography}
}
\end{document}